\documentclass[leqno]{amsart}
\usepackage{amsmath,amsfonts,amsthm,amssymb,indentfirst,epic,url}

\newcommand{\<}{\kern.0833em}

\newtheorem{theorem}{Theorem}
\newtheorem{lemma}[theorem]{Lemma}
\newtheorem{corollary}[theorem]{Corollary}
\newtheorem{proposition}[theorem]{Proposition}
\newtheorem{definition}[theorem]{Definition}

\newcommand{\U}{\mathcal{U}}
\newcommand{\card}{\mathrm{card}}
\newcommand{\supp}{\mathrm{supp}}
\newcommand{\fl}[1]{{\lfloor #1\rfloor}}
\newcommand{\strt}[1][1.7]{\vrule width0pt height0pt depth#1pt}

\begin{document}

\title[Maps on $k^I$ and product algebras]{Linear maps
on $k^I,$ and homomorphic images of infinite~direct~product~algebras}%
\thanks{After publication of this note, updates, errata, related
references etc., if found, will be recorded at
\url{http://math.berkeley.edu/~gbergman/papers/}
}

\subjclass[2010]{Primary: 15A04, 17A01, 17B05.
Secondary: 03C20, 03E55, 08B25, 17B20, 17B30.} 
\keywords{linear maps of $k^I$ to vector spaces of small dimension;
measurable cardinals;
homomorphisms on infinite direct products of nonassociative algebras;
simple, solvable, and nilpotent Lie algebras%
}

\author{George M. Bergman}
\address[G. Bergman]{University of California\\
Berkeley, CA 94720-3840, USA}
\email{gbergman@math.berkeley.edu}

\author{Nazih Nahlus}
\address[N. Nahlus]{American University of Beirut\\ %
Beirut, Lebanon}
\email{nahlus@aub.edu.lb}

\begin{abstract}
Let $k$ be an infinite field, $I$ an infinite set,
$V$ a $\!k\!$-vector-space, and $g:k^I\to V$ a $\!k\!$-linear map.
It is shown that if $\dim_k(V)$ is not too large
(under various hypotheses on $\card(k)$ and $\card(I),$
if it is finite, respectively less than $\card(k),$
respectively less than the continuum),
then $\ker(g)$ must contain elements $(u_i)_{i\in I}$
with all but finitely many components $u_i$ nonzero.

These results are used to prove that
every homomorphism from a direct product $\prod_I A_i$
of not-necessarily-associative algebras $A_i$ onto an algebra $B,$
where $\dim_k(B)$ is not too large (in the same senses) is
the sum of a map factoring
through the projection of $\prod_I A_i$ onto the product of
finitely many of the $A_i,$ and a map into
the ideal $\{b\in B\mid bB=Bb=\{0\}\}\subseteq B.$

Detailed consequences
are noted in the case where the $A_i$ are Lie algebras.

A version of the above result is also obtained with
the field $k$ replaced by a commutative valuation ring.
\end{abstract}
\maketitle

This note resembles \cite{prod_Lie1} in that the two papers obtain
similar results on homomorphisms on infinite product algebras;
but the methods are different, and the hypotheses under which the
methods of one note work
are in some ways stronger, in others weaker, than those of the other.
Also, in~\cite{prod_Lie1} we obtain many consequences from
our results, while here we aim for brevity, and
after one main result about general algebras, restrict
ourselves to a couple of quick consequences for Lie algebras.

The authors are grateful to Leo Harrington and Tom Scanlon for
helpful pointers to the literature, and to Jason Bell for
the strengthened version of Lemma~\ref{L.JPBell} used below.

\section{Definitions, and first results}\label{S.first}

Let us fix some terminology and notation.

\begin{definition}\label{D.alg&}
Throughout this note, $k$ will be a field.

By an {\em algebra} over $k$ we shall mean a
$\!k\!$-vector-space $A$ given with a $\!k\!$-bilinear
multiplication $A\times A\to A,$ which we do not
assume associative or unital.

If $A$ is an algebra, we define its {\em total annihilator ideal} to be
\begin{equation}\begin{minipage}[c]{35pc}\label{d.Z(A)}
$Z(A)~\,=~\,\{x\in A\mid xA=Ax=\{0\}\}.$
\end{minipage}\end{equation}

If $a=(a_i)_{i\in I}$ is an element of a direct product algebra
$A=\prod_I A_i,$ then we define its {\em support} as
\begin{equation}\begin{minipage}[c]{35pc}\label{d.supp}
$\supp(a)\ =\ \{i\in I\mid a_i\neq 0\}.$
\end{minipage}\end{equation}
For $J$ any subset of $I,$ we shall identify $\prod_{i\in J} A_i$
with the subalgebra of $\prod_{i\in I} A_i$ consisting of elements
whose support is contained in $J.$
We also define the subalgebra
\begin{equation}\begin{minipage}[c]{35pc}\label{d.Afin}
$A_\mathrm{fin}\ =\ \{a\in A\mid\supp(a)$ {\rm is finite}$\}.$
\end{minipage}\end{equation}
\end{definition}

The importance of $\!k\!$-linear functions on spaces $k^I$
to the study of homomorphisms on direct product algebras arises
from the following curious observation:

\begin{lemma}\label{L.supp}
Suppose $(A_i)_{i\in I}$ is a family of $\!k\!$-algebras,
$B$ a $\!k\!$-algebra,
$f:A=\prod_{i\in I}A_i\to B$ a surjective algebra homomorphism,
and $a=(a_i)_{i\in I}$ an element of $A,$ and consider the linear map
\begin{equation}\begin{minipage}[c]{35pc}\label{d.g_a}
$g_a:k^I\to B$\quad defined by\quad $g_a((u_i))=f((u_i a_i))$\quad
for all\quad $(u_i)\in k^I.$
\end{minipage}\end{equation}

Then\\[.2em]
\textup{(i)}~ If $\ker(g_a)$ contains an element $u=(u_i)_{i\in I}$
whose support is all of $I,$ then $f(a)\in Z(B).$\\[.2em]
\textup{(ii)}~ More generally, for any $u\in\ker(g_a),$ if we write
$a=a'+a'',$ where $\supp(a')\subseteq\supp(u)$ and
$\supp(a'')\subseteq I-\supp(u),$ then $f(a')\in Z(B).$\\[.2em]
\textup{(iii)}~ Hence, if $\ker(g_a)$ contains an element
whose support is cofinite in $I,$ then $a$ is the sum of an
element $a'\in f^{-1}(Z(B))$ and an element $a''\in A_\mathrm{fin}.$
\end{lemma}

\begin{proof}
(i): Given $u$ as in (i), and any $b\in B,$ let us write $b=f(x),$ where
$x=(x_i)\in A,$ and compute
\begin{equation}\begin{minipage}[c]{35pc}\label{d.fxa}
$f(a)\,b\ =\ f(a)f(x)\ =\ f(a\,x)\ =\ f((a_i\,x_i))\ =
\ f((u_i\,a_i\,u_i^{-1}\,x_i))\\[.2em]
\strt\quad=\ f((u_i\,a_i))\,f((u_i^{-1}\,x_i))\ =
\ 0\,f((u_i^{-1}\,x_i))\ =\ 0.$
\end{minipage}\end{equation}
So $f(a)$ left-annihilates all elements of $B;$ and by the
same argument with the order of factors reversed, it
right-annihilates all elements of $B.$
Thus, $f(a)\in Z(B),$ as claimed.

(ii):  Let $u'\in k^I$ be defined by taking $u'_i=u_i$ for
$i\in\supp(u),$ and $u'_i=1$ for $i\notin\supp(u).$
Thus, $\supp(u')=I;$ moreover, $u'a'=ua,$
whence $f(u'a')=f(ua)=0.$
Hence, $\ker(g_{a'})$
contains the element $u'$ whose support is $I;$
so by~(i), $f(a')\in Z(B).$

(iii) clearly follows from~(ii).
\end{proof}

Remark:  In the context of the above lemma, if the element
of $k^I$ having all entries equal to $1$ lies in $\ker(g_a),$ this
says that $f(a)=0.$
Part~(i) of the lemma says that, more generally, if an element with all
entries invertible lies in $\ker(g_a),$ then $f(a)$ is ``very
close to'' being zero.

Motivated by statement~(iii) of the lemma, let us look
for conditions under which the kernel
of a homomorphism on $k^I$ must contain elements of cofinite support.
Here is an easy one.

\begin{lemma}\label{L.poly}
Let $I$ be a set with $\card(I)\leq\card(k),$
and $g: k^I\to V$ a $\!k\!$-linear map, for some finite-dimensional
$\!k\!$-vector-space $V.$
Then there exists $u\in\ker(g)$ such that $I-\supp(u)$ is finite.
\end{lemma}

\begin{proof}
By the assumption on $\card(I),$ we can choose $x=(x_i)\in k^I$
whose entries $x_i$ are distinct.
Regarding $k^I$ as a $\!k\!$-algebra under componentwise operations,
let us map the polynomial algebra $k[t]$ into it by
the homomorphism sending $t$ to this~$x.$
Composing with $g:k^I\to V,$ we get a $\!k\!$-linear
map $k[t]\to V.$

Since $V$ is finite-dimensional, this map has nonzero kernel, so we
may choose $0\neq p(t)\in k[t]$ such that $p(x)\in\ker(g).$
Since the polynomial $p$ has only finitely many
roots, $p(x_i)$ is zero for only finitely many $i,$
so $p(x)$ gives the desired~$u.$
\end{proof}

Applying Lemma~\ref{L.poly} to maps $g_a$ as in Lemmas~\ref{L.supp},
and calling on statement~(iii) of the latter, we get

\begin{proposition}\label{P.first}
Let $k$ be an infinite field, let $(A_i)_{i\in I}$ be a family
of $\!k\!$-algebras such that the index set $I$
has cardinality $\leq\card(k),$ let
$A=\prod_I A_i,$ and let $f:A\to B$ be any surjective algebra
homomorphism to a finite-dimensional $\!k\!$-algebra $B.$

Then $B=f(A_{\mathrm{fin}})+Z(B).$
\textup{(}Equivalently, $A= A_{\mathrm{fin}}+f^{-1}(Z(B)).)$

Hence $B$ is the sum of $Z(B)$ and the \textup{(}mutually
annihilating\textup{)} images of finitely many of the $A_i.$
\end{proposition}

\begin{proof}
The first assertion follows immediately from the two preceding lemmas.
To get the final assertion, we note that since $B$ is
finite-dimensional, its subalgebra
$f(A_{\mathrm{fin}})=f(\bigoplus_I A_i)=\sum_I f(A_i)$
must be spanned by the images of finitely many of the $A_i,$
and since the $A_i,$ as subalgebras of $A,$ annihilate
one another, so do those images.
\end{proof}

In the next two sections we shall obtain three strengthenings of
Lemma~\ref{L.poly}, two of which weaken the assumption of
finite-dimensionality of $V,$ while the third, instead,
weakens the restriction on $\card(I).$

\section{Larger-dimensional $V.$}\label{S.dimB}

Our first generalization of Lemma~\ref{L.poly}
will be obtained by replacing the countable-dimensional
polynomial ring $k[t]$ by a subspace of the rational
function field $k(t)$ which has dimension $\card(k)$ over $k.$
Rational functions are not, strictly
speaking, functions; but that will be easy to fudge.

\begin{lemma}\label{L.rat}
For each $c\in k,$ let $p^{(c)}\in k^k$
be the function which for every $x\in k-\{c\}$ has
$p^{(c)}(x)=(x-c)^{-1},$ and at $c$ has the value $0.$
Then any nontrivial linear combination of the elements $p^{(c)}$
has at most finitely many zeroes.

Hence if $I$ is a set of cardinality $\leq\card(k),$ and $g$ is a
$\!k\!$-linear map of $k^I$ to
a $\!k\!$-vector-space $V$ of dimension $<\card(k),$ then $\ker(g)$
contains an element $u$ of cofinite support.
\end{lemma}

\begin{proof}
In $k(t),$ any linear combination of elements
$(t-c_1)^{-1},\ \dots\,,\ (t-c_n)^{-1}$ for distinct
$c_1,\dots,c_n\in k$ $(n\geq 1),$ such that each of these elements
has nonzero coefficient, gives a {\em nonzero} rational function
\begin{equation}\begin{minipage}[c]{35pc}\label{d.rat}
$a_1(t-c_1)^{-1}+\dots+\,a_n(t-c_n)^{-1}\ =
\ h(t)/((t-c_1)\dots(t-c_n))$\quad (where $h(t)\in k[t]).$
\end{minipage}\end{equation}
Indeed, to see that~(\ref{d.rat}) is nonzero
in $k(t),$ multiply by any $t-c_m.$
Then we can evaluate both sides at $c_m,$ and we find that
the left-hand side then has a unique nonzero term; so we must
have $h(c_m)\neq 0.$
Hence $h(t)$ is a nonzero element of $k[t],$
so~(\ref{d.rat}) is a nonzero element of~$k(t).$

If we now take the corresponding linear combination of
$p^{(c_1)},\,\dots\,,\,p^{(c_n)}$ in $k^k,$ the result has the value
$h(c)/((c-c_1)\dots(c-c_n))$ at each $c\neq c_1,\dots,c_n.$
Hence it is nonzero everywhere except at the finitely many
zeroes of $h(t),$ and some subset of the finite set $\{c_1,\dots,c_n\}.$

We get the final assertion by embedding the set $I$
in $k,$ so that the $p^{(c)}$ $(c\in k)$ induce elements of $k^I.$
These will form a $\!\card(k)\!$-tuple of functions,
any nontrivial linear combination of which is a function
with only finitely many zeroes.
Under a linear map $g$ from $k^I$ to a vector
space $V$ of dimension $<\card(k),$ some nontrivial linear
combination $u$ of these $\card(k)$ elements must go
to zero, yielding a member of $\ker(g)$ with the asserted property.
\end{proof}

(An alternative way to get around the problem that rational
functions have poles would be
to partition $k$ into two disjoint subsets of
equal cardinalities, and use linear combinations of rational functions
$1/(t-c)$ with $c$ ranging over
one of these sets to get functions on the other.)

For $k$ countable, the condition on the dimension of $V$ in the
final statement of the above lemma is no improvement
on what we got in Lemma~\ref{L.poly} using $k[t].$
In an earlier version of this note, we obtained an improvement
on Lemma~\ref{L.poly} for countable $k$ by a diagonal argument,
showing that if $k$ and $I$ are both countably infinite,
then a maximal subspace $W\subseteq k^I$
no nonzero member of which has infinitely many zero coordinates
must be uncountable-dimensional.
Jason Bell communicated to us the following stronger
result, which not only gives a subspace of continuum, rather than
merely uncountable, dimension, but (as is made clear in the proof,
though for simplicity we
do not include it in the statement), also shares with the constructions
of Lemmas~\ref{L.poly} and~\ref{L.rat} the property that
for every finite-dimensional subspace of $W,$ there is a uniform
bound on the number of zero coordinates of its nonzero elements,
which our earlier result lacked.
(The result below was, in fact, given in response to the question
we raised of whether
a construction admitting such uniform bounds was possible.)

\begin{lemma}[sketched by Jason Bell, personal communication]\label{L.JPBell}
If the field $k$ is infinite, and $I$ is a countably infinite set,
then there exists a subspace $W\subseteq k^I$
of continuum dimensionality such that
no nonzero member of $W$ has infinitely many \mbox{zeroes}.

Hence any $\!k\!$-linear map $g$ from $k^I$ to
a $\!k\!$-vector-space $V$ of less than continuum dimension
has in its kernel an element $u$ of cofinite support.
\end{lemma}

\begin{proof}
It suffices to prove the stated result for $I=\omega,$ the set of
natural numbers.

Let us first note that if $k$ is either of characteristic~$0,$ or
transcendental over its prime field, then it is algebraic
over a Unique Factorization Domain $R$ which is not a field
(namely, $\mathbb{Z},$ or a polynomial ring over the prime field
of $k).$
This ring $R$ admits a discrete valuation,
which induces a discrete valuation on the field of fractions of $R.$
It is easily deduced from \cite[Prop.~XII.4.2]{SL.Alg} that
this extends to a $\!\mathbb{Q}\!$-valued valuation $v$
on the algebraic extension $k$ of that field,
and by rescaling, $v$ can be assumed to have
valuation group containing $\mathbb{Z}.$
Let us call this situation Case~I.

If we are not in Case~I, then $k$ must be an infinite algebraic
extension of a finite field.
Hence it will contain a countable chain of distinct subfields,
\begin{equation}\begin{minipage}[c]{35pc}\label{d.k0k1}
$k_0\ \subset\ k_1\ \subset\ \cdots\ \subset\ k_i\ \subset\ \cdots~.$
\end{minipage}\end{equation}
Given any field $k$ containing such a chain of subfields
(regardless of characteristic, or algebraicity over a prime field),
we may define a natural-number-valued function $v$
(not a valuation) on $\bigcup_{i\in\omega} k_i\subseteq k$ by
letting $v(x)$ be the least $i$ such that $x\in k_i.$
We shall call the situation where $k$ contains
a chain~(\ref{d.k0k1}) Case~II.
(So Cases~I and~II together cover all infinite
fields, with a great deal of overlap.)

In either case, let us choose elements
$x_0,\ x_1,\ \dots\in k$ such that
\begin{equation}\begin{minipage}[c]{35pc}\label{d.vxi}
$v(x_i)=i$ \quad for all $i\in\omega,$
\end{minipage}\end{equation}
and for every real number $\alpha>1,$ let
$f_\alpha\in k^\omega$ be defined by
\begin{equation}\begin{minipage}[c]{35pc}\label{d.f_alpha}
$f_\alpha(n)=x_\fl{\alpha n}$ \quad $(n\in\omega),$
\end{minipage}\end{equation}
where $\fl{\alpha n}$ denotes the largest integer $\leq\alpha n.$

This gives continuum many elements $f_\alpha\in k^\omega.$
We shall now complete the proof by showing separately in Cases~I
and~II that for any
$1<\alpha_1<\dots<\alpha_d,$ there exists a natural number $N$
such that no nontrivial linear combination
\begin{equation}\begin{minipage}[c]{35pc}\label{d.c1f1+}
$c_1 f_{\alpha_1} +\dots + c_d f_{\alpha_d}$\quad $(c_1,\dots,c_d\in k)$
\end{minipage}\end{equation}
has more than $N$ zero coordinates.

If we are in Case~I, consider any $n$ such
that the $\!n\!$-th coordinate of~(\ref{d.c1f1+}) is zero.
This says that
\begin{equation}\begin{minipage}[c]{35pc}\label{d.sum_i=0}
$\sum_i c_i x_\fl{n\alpha_i}\ =\ 0.$
\end{minipage}\end{equation}
Now if a family of elements of $k$ which are not all zero has zero sum,
then at least two nonzero members of the family must have
equal valuation.
Thus, for some $i<j$ with $c_i, c_j\neq 0$ we have
\begin{equation}\begin{minipage}[c]{35pc}\label{d.v+v=v+v}
$v(c_i) + v(x_\fl{n\alpha_i})\ =\ v(c_j) + v(x_\fl{n\alpha_j}).$
\end{minipage}\end{equation}
By~(\ref{d.vxi}), this says
\begin{equation}\begin{minipage}[c]{35pc}\label{d.v+fl=v+fl}
$v(c_i) + \fl{n\alpha_i}\ =\ v(c_j) + \fl{n\alpha_j}.$
\end{minipage}\end{equation}
From the fact that $\fl{n\alpha_i}$ lies in the interval
$(n\alpha_i -1,\,n\alpha_i],$
and the corresponding fact for $\fl{n\alpha_j},$
we see that $\fl{n\alpha_i} - \fl{n\alpha_j}$ differs by
less than $1$ from $n\alpha_i - n\alpha_j,$
so~(\ref{d.v+fl=v+fl}) implies
\begin{equation}\begin{minipage}[c]{35pc}\label{d.n(*a-*a)}
$n(\alpha_j - \alpha_i)\ \in\ (v(c_i)-v(c_j)-1,\ v(c_i)-v(c_j)+1).$
\end{minipage}\end{equation}
This puts $n$ in an open interval of length $2/(\alpha_j-\alpha_i).$
We have shown that whenever the $\!n\!$-th coordinate of~(\ref{d.c1f1+})
is zero, this relation holds for some pair $i,j;$ so the total
number of possibilities for $n$ is at most
\begin{equation}\begin{minipage}[c]{35pc}\label{d.N=2sum}
$N\ =\ \sum_{i<j} \lceil 2/(\alpha_j-\alpha_i)\rceil,$
\end{minipage}\end{equation}
a bound depending only on $\alpha_1,\dots,\alpha_d$
(and not on $c_1,\dots,c_d),$ as claimed.

Next, suppose we are in Case~II.
Then we claim that for an element~(\ref{d.c1f1+}), there
can be at most $d-1$ values of $n$ with
\begin{equation}\begin{minipage}[c]{35pc}\label{d.n>}
$n\ \geq\ \max_{i=1,\dots,d-1}\,(1/(\alpha_{i+1}-\alpha_i))$
\end{minipage}\end{equation}
for which the $\!n\!$-th coordinate of~(\ref{d.c1f1+}) is zero.
For suppose, on the contrary, that $n_1<\dots<n_d$ all
have this property.
This says that the nonzero column vector of
coefficients $(c_1,\dots,c_d)^{\mathrm{T}}$ is
left annihilated by the $d\times d$ matrix
\begin{equation}\begin{minipage}[c]{35pc}\label{d.(())}
$((x_\fl{n_i \alpha_j})).$
\end{minipage}\end{equation}
Note that the subscripts $\fl{n_i\alpha_j}$ in~(\ref{d.(())})
are strictly increasing in both $i$ and $j;$ the former because
all $\alpha_j>1,$ the latter because all $n_i$ satisfy~(\ref{d.n>}).
It follows that in the matrix~(\ref{d.(())}), every minor
has the property that its lower right-hand entry does not lie
in the subfield generated by its other entries.
From this, it is easy to show by induction that all
minors have nonzero determinant, and so in particular
that~(\ref{d.(())}) is invertible.

But this contradicts the assumption that~(\ref{d.(())})
annihilates $(c_1,\dots,c_d)^{\mathrm{T}}.$
Hence there are, as claimed, at most $d-1$
values of $n$ satisfying~(\ref{d.n>})
such that the $\!n\!$-th entry of~(\ref{d.c1f1+}) is zero;
so the total number of zero entries of~(\ref{d.c1f1+}) is bounded by
\begin{equation}\begin{minipage}[c]{35pc}\label{d.N=max+d-1}
$N\ =\ \max_{i=1,\dots,d-1}\fl{1/(\alpha_{i+1}-\alpha_i)}+d,$
\end{minipage}\end{equation}
which again depends only on the $\alpha_i.$

The final assertion of the lemma clearly follows.
\end{proof}

Remark: In Case~I of the above proof, in place of
condition~(\ref{d.vxi})
we could equally well have used $x_i$ with $v(x_i)=-i.$
Similarly, the proof in Case~II can be adapted to fields $k$
having a {\em descending} chain of subfields
$k=k_0\supset k_1\supset\dots\supset k_i\supset\cdots~:$
in this situation, we define $v$ on $k-\bigcap_{i\in\omega} k_i$
to take each $x$
to the largest $i$ such that $x\in k_i,$ and consider upper
left-hand corners of minors instead of lower right-hand corners.
We know of no use for these observations at present; but they might be
of value in proving some variants of the above lemma.

\section{Larger $I.$}\label{S.I}
For our third generalization of Lemma~\ref{L.poly}, we
return to the hypothesis that $V$ is finite-dimensional, and
prove that in that situation, the statement that every linear map
$g:k^I\to V$ has elements of cofinite support in fact holds
for sets $I$ of cardinality much greater than $\card(k).$

We can no longer get this conclusion by finding an
infinite-dimensional subspace $W\subseteq k^I$
whose nonzero members each have only finitely many zeroes.
On the contrary, when $\card(I)>\card(k)$ (with the former
infinite) there can be no subspace $W\subseteq k^I$ of
dimension $>1$ whose nonzero members all have only finitely many zeroes.
For if $(x_i)$ and $(y_i)$ are linearly independent elements of
$W,$ and we look at the subspaces of $k^2$
generated by the pairs $(x_i,y_i)$ as $i$ runs over $I,$
then if $\card(I)>\card(k),$
at least one of these subspaces must occur at $\card(I)$ many values
of $i,$ but cannot occur at all $i;$ hence some
linear combination of $(x_i)$ and $(y_i)$ will have $\card(I)$ zeroes,
but not itself be zero.
So we must construct our elements of cofinite support in a different
way, paying attention to the particular map~$g.$

Surprisingly, our proof will again
use the polynomial trick of Lemma~\ref{L.poly};
though this time only after considerable preparation.
(We could use rational functions in place of these polynomials as in
Lemma~\ref{L.rat}, or functions like
the $f_\alpha$ of Lemma~\ref{L.JPBell}, but so far
as we can see, this would not improve our result, since
finite-dimensionality of $V$ is required by other
parts of the argument.)

The case of Theorem~\ref{T.main} below that we will deduce
from the result of this section is actually slightly weaker than the
corresponding result proved by different methods in~\cite{prod_Lie1}.
Hence the reader who is only interested in consequences for
algebra homomorphisms $\prod_I A_i\to B$ may prefer to skip
the lengthy and intricate argument of this section.
On the other hand, insofar as our general technique makes the
question, ``For what $k,$ $I$ and $V$ can we say that the kernel of
every $\!k\!$-linear map $k^I\to V$ must contain an element of cofinite
support?''\ itself of interest, the result of this section creates a
powerful complement to those of the preceding section.

We will assume here familiarity with the definitions of ultrafilter and
ultraproduct (given in most books on universal algebra or model theory,
and summarized in \cite[\S14]{prod_Lie1}), and
of $\!\kappa\!$-completeness of an ultrafilter
(developed, for example, in \cite{Ch+Keis} or \cite{Drake},
and briefly summarized in
the part of \cite[\S15]{prod_Lie1} preceding Theorem~47).

In the lemma below, we do not yet restrict $\card(I)$ at all.
As a result, we will get functions with zero-sets
characterized in terms of finitely many $\!\card(k)^+\!$-complete
ultrafilters, rather than finitely many points.
In the corollary to the lemma, we add a cardinality restriction
which forces such ultrafilters to be principal,
and so get elements with only finitely many zeroes.
The lemma also allows $k$ to be finite, necessitating
a proviso~(\ref{d.dim+2}) that its cardinality not be too small
relative to $\dim_k(V);$ this, too, will go away in the corollary,
where, for other reasons, we will have to require $k$ to be infinite.

In reading the lemma and its proof, the reader might bear in mind that
the property~(\ref{d.J0}) makes $J_0$ ``good'' for our purposes,
while $J_1,\dots,J_n$ embody the complications that we must overcome.
The case of property~(\ref{d.J0}) that we will want in the end
is for the element $0\in g(k^{J_0});$ but in the course of
the proof it will be important to consider that property for
arbitrary elements of that subspace.

\begin{lemma}\label{L.ultra}
Let $I$ be a set, $V$ a finite-dimensional
$\!k\!$-vector space such that
\begin{equation}\begin{minipage}[c]{35pc}\label{d.dim+2}
$\card(k)\ \geq\ \dim_k(V)+2,$
\end{minipage}\end{equation}
and $g:k^I\to V$ a $\!k\!$-linear map.

Then $I$ may be decomposed into finitely many disjoint subsets,
\begin{equation}\begin{minipage}[c]{35pc}\label{d.I=}
$I\ =\ J_0\cup J_1\cup\,\dots\,\cup J_n$
\end{minipage}\end{equation}
$(n\geq 0),$ such that
\begin{equation}\begin{minipage}[c]{35pc}\label{d.J0}
every element of $g(k^{J_0})$ is the image under $g$ of an element
having support precisely $J_0,$
\end{minipage}\end{equation}
and such that each set $J_m$ for $m=1,\dots,n$ has on it a
$\!\card(k)^+\!$-complete ultrafilter $\U_m$
such that, letting $\psi$ denote the factor-map
$V\to V/g(k^{J_0}),$ the composite $\psi g: k^I\to V/g(k^{J_0})$
can be factored
\begin{equation}\begin{minipage}[c]{35pc}\label{d.gfactors}
$k^I\ =\ k^{J_0}\times k^{J_1}\times\dots\times k^{J_n}\ %
\to\ k^{J_1}/\,\U_1\times\dots\times k^{J_n}/\,\U_n\ %
\hookrightarrow\ V/g(k^{J_0}),$
\end{minipage}\end{equation}
where $k^{J_m}/\,\U_m$ denotes the ultrapower of $k$ with respect
to the ultrafilter $\U_m,$
the first arrow of~\textup{(\ref{d.gfactors})}
is the product of the natural projections,
and the last arrow is an embedding.
\end{lemma}

\begin{proof}
If $\card(k)=2,$ then~(\ref{d.dim+2}) makes $V=\{0\},$ and
the lemma is trivially true (with $J_0=I$ and $n=0);$
so below we may assume $\card(k)>2.$

There exist subsets $J_0\subseteq I$ satisfying~(\ref{d.J0});
for instance, the empty subset.
Since $V$ is finite-dimensional, we may choose
a $J_0$ satisfying~(\ref{d.J0}) such that
\begin{equation}\begin{minipage}[c]{35pc}\label{d.max}
Among subsets of $I$ satisfying~(\ref{d.J0}), $J_0$ maximizes the
subspace $g(k^{J_0})\subseteq V,$
\end{minipage}\end{equation}
i.e., such that no subset $J'_0$ satisfying~(\ref{d.J0}) has
$g(k^{J'_0})$ properly larger than $g(k^{J_0}).$

Given this $J_0,$ we now consider subsets $J\subseteq I-J_0$ such that
\begin{equation}\begin{minipage}[c]{35pc}\label{d.Jmin}
$g(k^J)\not\subseteq g(k^{J_0}),$ and $J$ minimizes the
subspace $g(k^{J_0})+g(k^J)$ subject to this\\
condition, in the sense that every
subset $J'\subseteq J$ satisfies either
\end{minipage}\end{equation}
\begin{equation}\begin{minipage}[c]{35pc}\label{d.small}
$g(k^{J'})\ \subseteq\ g(k^{J_0})$
\end{minipage}\end{equation}
or
\begin{equation}\begin{minipage}[c]{35pc}\label{d.big}
$g(k^{J_0})+g(k^{J'})\ =\ g(k^{J_0})+g(k^J).$
\end{minipage}\end{equation}

It is not hard to see from the finite-dimensionality
of $V,$ and the fact that inclusions of sets $J$ imply the
corresponding inclusions among the
subspaces $g(k^{J_0})+g(k^{J}),$ that such
minimizing subsets $J$ will exist if $g(k^{J_0})\neq g(k^I).$
If, rather, $g(k^{J_0})=g(k^I),$ then the collection of
such subsets that we develop in the arguments below will be
empty, but that will not be a problem.

Let us, for the next few paragraphs, fix such a $J.$
Thus, every $J'\subseteq J$ satisfies
either~(\ref{d.small}) or~(\ref{d.big}).
However, we claim that there cannot be many {\em pairwise disjoint}
subsets $J'\subseteq J$ satisfying~(\ref{d.big}).
Precisely, letting
\begin{equation}\begin{minipage}[c]{35pc}\label{d.e=}
$e\ =\ \dim_k((g(k^{J_0})+g(k^J))/g(k^{J_0})),$
\end{minipage}\end{equation}
we claim that there cannot be $2e$ such pairwise disjoint subsets.

For suppose we had pairwise disjoint sets $J'_{\alpha, d}\subseteq J$
$(\alpha \in\{0,1\},\ d\in\{1,\dots,e\})$ each satisfying~(\ref{d.big}).
Let
\begin{equation}\begin{minipage}[c]{35pc}\label{d.h1he}
$h_1,\dots,h_e\ \in\ g(k^{J_0})+g(k^J)$
\end{minipage}\end{equation}
be a minimal family spanning $g(k^{J_0})+g(k^J)$ over $g(k^{J_0}).$
For each $\alpha\in\{0,1\}$ and $d\in\{1,\dots,e\},$
condition~(\ref{d.big}) on $J'_{\alpha, d}$
allows us to choose an element
$x^{(\alpha,d)}\in k^{J'_{\alpha, d}}$ such that
\begin{equation}\begin{minipage}[c]{35pc}\label{d.x*ad}
$g(x^{(\alpha,d)})\ \equiv\ h_d\ \ (\mathrm{mod}\ g(k^{J_0})).$
\end{minipage}\end{equation}
Some of the $x^{(\alpha,d)}$ may have support strictly smaller
than the corresponding set $J'_{\alpha, d};$ if this happens,
let us cure it by replacing
$J'_{\alpha, d}$ by $\supp(x^{(\alpha,d)}):$ these are still pairwise
disjoint subsets of $J,$ and will still satisfy~(\ref{d.big}) rather
than~(\ref{d.small}), since after this modification, the
subspace $g(k^{J'_{\alpha, d}})$ still contains
$g(x^{(\alpha,d)})\notin g(k^{J_0}).$

We now claim that the set
\begin{equation}\begin{minipage}[c]{35pc}\label{d.J*0}
$J^*_0\ =\ J_0\,\cup\,
\bigcup_{\alpha\in\{0,1\},\ d\in\{1,\dots,e\}} J'_{\alpha, d}$
\end{minipage}\end{equation}
contradicts the maximality condition~(\ref{d.max}) on $J_0.$
Clearly $g(k^{J^*_0})=g(k^{J_0})+g(k^J)$ is strictly
larger than $g(k^{J_0}).$
To show that $J^*_0$ satisfies the analog of~(\ref{d.J0}),
consider any $h\in g(k^{J^*_0})=g(k^{J_0})+g(k^J),$
and let us write it, using the relative spanning set~(\ref{d.h1he}), as
\begin{equation}\begin{minipage}[c]{35pc}\label{d.h0}
$h\ =\ h_0+c_1 h_1+\dots+c_e h_e
\quad (h_0\in g(k^{J_0}),\ c_1,\dots,c_e\in k).$
\end{minipage}\end{equation}
Since $\card(k)>2,$ we can now choose for each $d=1,\dots,e$ an element
$c'_d\in k$ which is neither $0$ nor $c_d,$ and form the element
\begin{equation}\begin{minipage}[c]{35pc}\label{d.cc'}
$x\ =\ (c'_1 x^{(0,1)} + (c_1-c'_1) x^{(1,1)})\,+\,
(c'_2 x^{(0,2)} + (c_2-c'_2) x^{(1,2)})\,+\,\cdots\,+\,
(c'_e x^{(0,e)} + (c_e-c'_e) x^{(1,e)}).$
\end{minipage}\end{equation}
By our choice of $c'_1,\dots,c'_e,$ none of the coefficients
$c'_d$ or $c_d-c'_d$ is zero, so $\supp(x)=\bigcup J'_{\alpha, d}.$
Applying $g$ to~(\ref{d.cc'}), we see
from~(\ref{d.x*ad}) that $g(x)$ is congruent modulo
$g(k^{J_0})$ to $c_1 h_1+\dots+c_e h_e,$ hence,
by~(\ref{d.h0}), congruent to $h.$
By~(\ref{d.J0}), we can find an element $y\in k^{J_0}$ with support
precisely $J_0$ that makes up the difference, so that $g(y)+g(x)=h.$
The element $y+x$ has support exactly $J^*_0;$ and since we have
obtained an arbitrary $h\in g(k^{J^*_0})$
as the image under $g$ of this element, we have shown
that $J^*_0$ satisfies the analog of~(\ref{d.J0}),
giving the desired contradiction.

Thus, we have a finite upper bound (namely, $2e-1)$
on the number of pairwise disjoint subsets $J'$ that
$J$ can contain which satisfy~(\ref{d.big}).
So starting with $J,$ let us, if it is the union of two disjoint
subsets with that property, split one off and
rename the other $J,$ and repeat this process as many times as we can.
Then in finitely many steps, we must
get a $J$ which cannot be further decomposed.
Summarizing what we know about this $J,$ we have
\begin{equation}\begin{minipage}[c]{35pc}\label{d.J}
$g(k^J)\not\subseteq g(k^{J_0}),$
every subset $J'\subseteq J$ satisfies either
$g(k^{J'})\subseteq g(k^{J_0})$ or
$g(k^{J_0})+g(k^{J'})=g(k^{J_0})+g(k^J),$
and no two {\em disjoint} subsets of $J$ satisfy the latter equality.
\end{minipage}\end{equation}

Let us call any subset $J\subseteq I-J_0$ satisfying~(\ref{d.J})
a {\em nugget}.
From the above development, we see that
\begin{equation}\begin{minipage}[c]{35pc}\label{d.nuggets}
Every subset $J\subseteq I-J_0$ such that
$g(k^J)\not\subseteq g(k^{J_0})$ contains a nugget.
\end{minipage}\end{equation}

The rest of this proof will analyze the properties of an arbitrary
nugget $J,$ and finally show (after a possible adjustment
of $J_0)$ that $I-J_0$ can be decomposed into
finitely many nuggets $J_1\cup\dots\cup J_n,$ and that these will have
the properties in the statement of the proposition.

We begin by showing that
\begin{equation}\begin{minipage}[c]{35pc}\label{d.ultra}
If $J$ is a nugget, then the set
$\U=\{J'\subseteq J\mid g(k^{J_0})+g(k^{J'})=g(k^{J_0})+g(k^J)\}$\\
is an ultrafilter on $J_{\strt}.$
\end{minipage}\end{equation}
To see this, note that by~(\ref{d.J}), the complement of $\U$
within the set of subsets of $J$
is also the set of complements relative to $J$ of members
of $\U,$ and is, furthermore, the set of all $J'\subseteq J$ such
that $g(k^{J'})\subseteq g(k^{J_0}).$
The latter set is clearly closed under unions and passing to smaller
subsets, hence $\U,$ inversely, is closed
under intersections and passing to larger subsets of $J;$
i.e., $\U$ is a filter.
Since $\emptyset\notin\U,$
while the complement of any subset of $J$ not in $\U$ does belong
to $\U,$ $\U$ is an ultrafilter.

Let us show next that any nugget $J$ has properties
that come perilously close to making $J_0\cup J$ a counterexample
to the maximality condition~(\ref{d.max}) on $J_0.$
By assumption, $g(k^{J_0\cup J})$ is strictly larger than $g(k^{J_0}).$
Now consider any $h\in g(k^{J_0\cup J}).$
We may write
\begin{equation}\begin{minipage}[c]{35pc}\label{d.h=}
$h\ =\ g(w)+g(x),$\quad where $w\in k^{J_0},\ x\in k^J.$
\end{minipage}\end{equation}
Suppose first that
\begin{equation}\begin{minipage}[c]{35pc}\label{d.hnotin}
$h\notin g(k^{J_0}).$
\end{minipage}\end{equation}
From~(\ref{d.h=}) and~(\ref{d.hnotin})
we see that $g(x)\notin g(k^{J_0}),$ so $\supp(x)\in\U.$
Now take any element $x'\in k^J$ which
agrees with $x$ on $\supp(x),$ and has (arbitrary) {\em nonzero}
values on all points of $J-\supp(x).$
The element by which we have modified $x$ to get $x'$ has
support in $J-\supp(x),$ which is $\notin\U$ because $\supp(x)\in\U;$
hence $g(x')\equiv g(x)\ (\mathrm{mod}\ g(k^{J_0})),$ hence
by~(\ref{d.h=}), $g(x')\equiv h\ (\mathrm{mod}\ g(k^{J_0})).$
Hence by~(\ref{d.J0}), we can find $z\in k^{J_0}$ with
support exactly $J_0$ such that $g(z)+g(x')=h.$
Thus, $z+x'$ is an element with support
$J_0\cup J$ whose image under $g$ is $h.$

This is just what would be needed to make $J_0\cup J$
satisfy~(\ref{d.J0}), if we had proved it for all
$h\in g(k^{J_0\cup J});$ but we have only proved it for $h$
satisfying~(\ref{d.hnotin}) (which we needed to argue
that $\supp(x)$ belonged to $\U).$

We now claim that if there were any $x\in k^J$ with $\supp(x)\in\U$
satisfying $g(x)\in g(k^{J_0}),$ then we would be able to
complete our argument contradicting~(\ref{d.max}).
For modifying such an $x$ by any element with complementary support
in $J,$ we would get an element with support exactly $J$
whose image under $g$ would still lie in $g(k^{J_0}).$
Adding to this element the images under $g$ of
all elements of $k^{J_0}$ with support equal to $J_0,$ we would get
images under $g$ of certain elements with support exactly $J_0\cup J.$
Moreover, since $J_0$ satisfies~(\ref{d.J0}), these sums
would comprise all $h\in g(k^{J_0}),$ i.e., just those values that were
excluded by~(\ref{d.hnotin}).
In view of the resulting contradiction to~(\ref{d.max}), we have proved
\begin{equation}\begin{minipage}[c]{35pc}\label{d.bigsupp}
If $J$ is a nugget, then every $x\in k^J$ with
$\supp(x)\in\U$ satisfies $g(x)\notin g(k^{J_0}).$
\end{minipage}\end{equation}

We shall now use the ``polynomial functions'' trick to show
that~(\ref{d.bigsupp})
can only hold if the ultrafilter $\U$ is $\!\card(k)^+\!$-complete.
If $k$ is finite, $\!\card(k)^+\!$-completeness is vacuous,
so assume for the remainder of this paragraph that $k$ is infinite.
If $\U$ is not $\!\card(k)^+\!$-complete, we can
find pairwise disjoint subsets $J_c\subseteq J$ $(c\in k)$
with $J_c\notin\U,$ whose union is all of $J.$
Given these subsets, let $z\in k^J$ be the element
having, for each $c\in k,$ the value $z_i=c$ at all $i\in J_c.$
Taking powers of $z$ under componentwise multiplication,
we get elements $1,\,z,\dots,z^n,\ldots\in k^J.$
Since $V$ is finite-dimensional, some nontrivial linear combination
$p(z)$ of these must be in the kernel of $g.$
But as a nonzero polynomial, $p$ has only finitely many
roots in $k,$ say $c_1,\dots,c_r.$
Thus $\supp(p(z))=J-(J_{c_1}\cup\dots\cup J_{c_r}).$
Since $J\in\U$ and $J_{c_1},\dots, J_{c_r}\notin\U,$
we get $\supp(p(z))\in\U;$ but since $p(z)\in\ker(g),$
we have $g(p(z))=0\in g(k^{J_0}),$ contradicting~(\ref{d.bigsupp}).
Hence
\begin{equation}\begin{minipage}[c]{35pc}\label{d.card+}
For every nugget $J,$ the ultrafilter $\U$ of~(\ref{d.ultra})
is $\!\card(k)^+\!$-complete.
\end{minipage}\end{equation}

We claim next that~(\ref{d.card+}) implies that for any nugget $J,$
\begin{equation}\begin{minipage}[c]{35pc}\label{d.dim1}
$\dim_k((g(k^{J_0})+g(k^J))/g(k^{J_0}))\ =\ 1.$
\end{minipage}\end{equation}
Indeed, fix $x\in k^J$ with support $J,$ and
consider any $y\in k^J.$
If we classify the elements $i\in J$ according to the
value of $y_i/x_i\in k,$ this gives $\card(k)$ sets, so
by $\!\card(k)^+\!$-completeness, one
of them, say $\{i\mid y_i=c\,x_i\}$ (for some $c\in k)$ lies in $\U.$
Hence $y-c\,x$ has support $\notin\U,$ so $g(y-c\,x)\in g(k^{J_0}),$
i.e., modulo $g(k^{J_0}),$ the element
$g(y)$ is a scalar multiple of $g(x).$
So $g(x)$ spans $g(k^{J_0})+g(k^J)$ modulo~$g(k^{J_0}).$

Let us now choose for each nugget $J$ an element $x_J$ with support $J.$
Thus, by the above observations, $g(x_J)$ spans
$g(k^{J_0})+g(k^J)$ modulo $g(k^{J_0}).$
We claim that
\begin{equation}\begin{minipage}[c]{35pc}\label{d.indep}
For any disjoint nuggets $J_1,\dots,J_n,$ the elements
$g(x^{\strt}_{J_1}),\dots,g(x^{\strt}_{J_n})\in V$ are linearly\\
independent modulo $g(k^{J_0}).$
\end{minipage}\end{equation}

For suppose, by way of contradiction, that we had some relation
\begin{equation}\begin{minipage}[c]{35pc}\label{d.linrel}
$\sum_{m=1}^n c_m\,g(x^{\strt}_{J_m})\ \in\ g(k^{J_0}),$\quad
with not all $c_m$ zero.
\end{minipage}\end{equation}
If $n>\dim_k(V),$ then there must be a linear relation in $V$ among
$\leq \dim_k(V)+1$ of the $g(x^{\strt}_{J_m})\in V,$ so in that
situation we may (in working toward our contradiction) replace
the set of nuggets assumed to satisfy a relation~(\ref{d.linrel})
by a subset also satisfying
\begin{equation}\begin{minipage}[c]{35pc}\label{d.nleq}
$n\ \leq\ \dim_k(V)+1,$
\end{minipage}\end{equation}
and~(\ref{d.linrel}) by a relation which they satisfy.
Also, by dropping from our list of nuggets in~(\ref{d.linrel})
any $J_m$ such that $c_m=0,$
we may assume those coefficients all nonzero.

We now invoke for the third (and last) time the maximality
assumption~(\ref{d.max}), arguing that in the above
situation, $J_0\cup J_1\cup\dots\cup J_n$ would be a counterexample
to that maximality.

For consider any
\begin{equation}\begin{minipage}[c]{35pc}\label{d.v}
$v\ \in\ g(k^{J_0\,\cup\,J_1\,\cup\,\dots\,\cup\,J_n}).$
\end{minipage}\end{equation}
By~(\ref{d.dim1}) and our choice of $x_{J_1},\dots,x_{J_n},$ $v$ can
be written as the sum of an element of $g(k^{J_0})$ and
an element $\sum d_m\,g(x^{\strt}_{J_m})$ with $d_1,\dots,d_n\in k.$
By~(\ref{d.dim+2}) and~(\ref{d.nleq}), $\card(k)\geq\dim_k(V)+2>n,$
hence we can choose an element $c\in k$ distinct from
each of $d_1/c_1,$ $\dots\,,$ $d_n/c_n$ (for
the $c_m$ of~(\ref{d.linrel})), i.e.,
such that $d_m-c\,c_m\neq 0$ for $m=1,\dots,n.$
Thus, $\sum\,(d_m-c\,c_m)\,x^{\strt}_{J_m},$
which by~(\ref{d.linrel}) has the same
image in $V/g(k^{J_0})$ as our given element $v,$ is a linear
combination of $x^{\strt}_{J_1},\dots,x^{\strt}_{J_n}$ with
{\em nonzero} coefficients,
hence has support exactly $J_1\cup\dots\cup J_n.$
As before, we can now use~(\ref{d.J0}) to adjust this by an
element with support exactly $J_0$ so that the image under $g$ of the
resulting element $x$ is $v.$
Since $x$ has support exactly $J_0\cup J_1\cup\dots\cup J_n,$
we have the desired contradiction to~(\ref{d.max}).

It follows from~(\ref{d.indep}) that
there cannot be more than $\dim_k(V)$ disjoint nuggets;
so a maximal family of pairwise disjoint nuggets will be finite.
Let $J_1,\dots,J_n$ be such a maximal family.

In view of~(\ref{d.nuggets}), the set
$J=I-(J_0\cup J_1\cup\dots\cup J_n)$
must satisfy $g(k^J)\subseteq g(k^{J_0}),$ hence we can
enlarge $J_0$ by adjoining to it that set $J,$
without changing $g(k^{J_0}),$ and hence without losing~(\ref{d.J0}).
We then have~(\ref{d.I=}).

For $m=1,\dots,n,$
let $\U_m$ be the ultrafilter on $J_m$ described in~(\ref{d.ultra}).
To verify the final statement of the proposition, that there exists
a factorization~(\ref{d.gfactors}), note that any element
of $k^I$ can be written $a^{(0)}+a^{(1)}+\dots+a^{(n)}$ with
$a^{(m)}\in k^{J_m}$ $(m=0,\dots,n),$ hence its image under $g$ will be
congruent modulo $g(k^{J_0})$ to $g(a^{(1)})+\dots+g(a^{(n)}).$
Now the image of each $g(a^{(m)})$ modulo $g(k^{J_0})$
is a function only of the equivalence class of $a^{(m)}$ with respect to
the ultrafilter $\U_m$ (since two elements in the same equivalence
class will disagree on a subset of
$J_m$ that is $\notin\U_m,$ so that their difference
is mapped by $g$ into $g(k^{J_0})).$
Hence the value of $g(a)$ modulo $g(k^{J_0})$
is determined by the images of $a$ in the spaces $k^{J_m}/\,\U_m.$
This gives the factorization~(\ref{d.gfactors}).
The one-one-ness of the factoring map follows from~(\ref{d.indep}).
\end{proof}

To get from this a result with a simpler statement, recall that a set
$I$ admits a {\em nonprincipal} $\!\card(k)^+\!$-complete
ultrafilter only if its cardinality is greater than or equal to a
measurable cardinal $>\card(k)$ \cite[Proposition~4.2.7]{Ch+Keis}.
(We follow~\cite{Ch+Keis} in counting $\aleph_0$ as a measurable
cardinal.
Thus, we write ``uncountable measurable cardinal''
for what many authors, e.g., \cite[p.177]{Drake}, simply call a
``measurable cardinal''.)

Now uncountable measurable cardinals, if they exist at all,
must be enormously large (cf.\ \cite[Chapter~6, Corollary~1.8]{Drake}).
Hence for $k$ infinite, it is a weak restriction to assume
that $I$ is smaller than all $\!\card(k)^+\!$-complete cardinals.
Under that assumption, the $\!\card(k)^+\!$-complete ultrafilters
$\U_m$ of Lemma~\ref{L.ultra}
must be principal, determined by elements $i_m\in I;$
so each nugget $J_m$ contains a minimal nugget, the singleton $\{i_m\},$
and we may use these minimal nuggets in our
decomposition~(\ref{d.I=}).
The statement of Lemma~\ref{L.ultra} then simplifies to the next result.
(No such simplification is possible if $k$ is finite,
since then every ultrafilter is $\!\card(k)^+\!$-complete,
and the only restriction we could put on $\card(I)$ that would force
all $\!\card(k)^+\!$-complete ultrafilters to be
principal would be finiteness; an uninteresting situation.
So we now exclude the case of finite~$k.)$

\begin{corollary}\label{C.<meas}
Let $k$ be an infinite field, $I$ a set of cardinality less than
every measurable cardinal $>\nolinebreak\card(k)$
\textup{(}if any exist\textup{)}, $V$ a finite-dimensional
$\!k\!$-vector space, and $g:k^I\to V$ a $\!k\!$-linear map.
Then there exist elements $i_1,\dots,i_n\in I$
such that, writing $J_0=I-\{i_1,\dots,i_n\},$ we have
\begin{equation}\begin{minipage}[c]{35pc}\label{d.I-i*}
Every element of $g(k^{J_0})$ is the image under
$g$ of an element having support precisely $J_0.$
\end{minipage}\end{equation}

In particular, applying this to $0\in g(k^{J_0}),$
\begin{equation}\begin{minipage}[c]{35pc}\label{d.fewzeros}
There exists some $u=(u_i)\in\ker(g)$ such
that $u_i=0$ for only finitely many $i$
\textup{(}namely $i_1,\dots,i_n).$
\qed\hspace{-1.3em}
\end{minipage}\end{equation}
\end{corollary}

Since we have excluded the case where $k$ is finite, the
above corollary did not need condition~(\ref{d.dim+2}),
that $\card(k)\geq\dim_k(V)+2.$
We end this section with a quick example showing that
Lemma~\ref{L.ultra} does need that condition.

Let $k$ be any finite field, and $I$ a subset of $k\times k$
consisting of one nonzero element from each of the
$\card(k)+1$ one-dimensional subspaces of that
two-dimensional space (i.e., $I$ is a set of
representatives of the points of the projective line over $k).$
Let $S\subseteq k^I$ be the two-dimensional subspace consisting of the
restrictions to $I$ of all $\!k\!$-linear functionals on $k\times k.$
Since $k^I$ is $\!(\card(k){+}1)\!$-dimensional, $S$ can be expressed
as the kernel of a linear map $g$ from $k^I$ to a
$\!(\card(k){-}1)\!$-dimensional vector space $V.$

By choice of $I,$ every element of $S=\ker(g)$
has a zero somewhere on $I,$ so
$0\in g(k^I)$ is not the image under $g$ of an element having
all of $I$ for support.
Hence~(\ref{d.J0}) cannot hold with $J_0=I.$
If Lemma~\ref{L.ultra} were applicable, this would force
the existence of a nonzero number of nuggets $J_m.$
Since $I$ is finite, the associated
ultrafilters would be principal, corresponding to elements $i_m$ such
that all members of $S=\ker(g)$ were zero at $i_m$
(by the one-one-ness of the last map of~(\ref{d.gfactors})).
But this does not happen either: for every $i\in I,$
there are clearly elements of $S$ nonzero at $i.$

Hence the conclusion of Lemma~\ref{L.ultra} does not hold for this $g.$
Note that since $\dim_k(V)=\card(k)-1,$
the condition $\card(k)\geq\dim_k(V)+2$ fails by just~$1.$

\section{Back to homomorphic images of product algebras}\label{S.*PA_i->B}

From the above three results on elements with cofinite
support, we can now prove the three cases of

\begin{theorem}\label{T.main}
Assume the field $k$ is infinite, and let $(A_i)_{i\in I}$ be a family
of $\!k\!$-algebras, $B$ a $\!k\!$-algebra, and
$f:A=\prod_I A_i\to B$ a surjective $\!k\!$-algebra homomorphism.

Suppose further that either\\[.2em]
\textup{(i)}\ \ $\dim_k(B)<\card(k),$ and
$\card(I)\leq\card(k),$ or\\[.2em]
\textup{(ii)}\,\ $\dim_k(B)<2^{\aleph_0},$ and $\card(I)=\aleph_0,$
or\\[.2em]
\textup{(iii)}\ $\dim_k(B)$ is finite, and $\card(I)$ is less than
every measurable cardinal $>\card(k).$\vspace{.2em}

Then
\begin{equation}\begin{minipage}[c]{35pc}\label{d.B=}
$B\ =\ f(A_{\mathrm{fin}})+Z(B).$
\end{minipage}\end{equation}

In fact, $f$ can be written as the sum
$f_1+f_0$ of a $\!k\!$-algebra homomorphism $f_1:A\to B$ that
factors through the projection of $A$ onto the product
of finitely many of the $A_i,$ and a $\!k\!$-algebra
homomorphism $f_0:A\to Z(B).$
\end{theorem}

\begin{proof}
We see~(\ref{d.B=}) by combining Lemma~\ref{L.supp}(iii)
with Lemma~\ref{L.rat} in case~(i),
with Lemma~\ref{L.JPBell} in case~(ii), and
with Corollary~\ref{C.<meas} in case~(iii).
The remainder of the proof will be devoted to establishing
the final assertion.

In doing so, let us identify each
algebra $A_{i_0}$ $(i_0\in I)$ with the subalgebra of $A$ consisting
of elements with support in $\{i_0\}.$
In particular, given $a=(a_i)\in A$ and $i_0\in I,$ the component
$a_{i_0}\in A_{i_0}$ will also be regarded as an element of $A.$

As in the proof of Proposition~\ref{P.first},
$f(A_{\mathrm{fin}})=\sum_I f(A_i);$ but since not all of our
alternative hypotheses~(i)--(iii)
have $B$ finite-dimensional, we need a new argument
to show that only finitely many of these
summands are needed in~(\ref{d.B=}).
In fact, we shall show that the set
\begin{equation}\begin{minipage}[c]{35pc}\label{d.I1}
$I_1\ =\ \{i\in I\mid f(A_i)\not\subseteq Z(B)\}$
\end{minipage}\end{equation}
is finite.
To this end, let us choose for each $i\in I_1$ an $a_i\in A_i$
such that $f(a_i)\notin Z(B),$ and
let $a=(a_i)_{i\in I_1}\in \prod_{I_1} A_i\subseteq\prod_I A_i.$
By~(\ref{d.B=}), there exist
$a'\in A_{\mathrm{fin}}$ and $z\in Z(B)$ such that
\begin{equation}\begin{minipage}[c]{35pc}\label{d.fa'+z}
$f(a)\ =\ f(a')+z.$
\end{minipage}\end{equation}
We claim that $I_1\subseteq\supp(a').$
Indeed, consider any $i_1\in I_1.$
Since $f(a_{i_1})\notin Z(B),$ we can find some
element of $B,$ which we write $f(x),$ where $x=(x_i)\in A,$ such
that either $f(x)f(a_{i_1})\neq 0$ or $f(a_{i_1})f(x)\neq 0.$
Without loss of generality, assume the latter inequality.
Then
\begin{equation}\begin{minipage}[c]{35pc}\label{d.ai0x}
$0\ \neq\ f(a_{i_1})f(x)\ =\ f(a_{i_1}x)\ =\ f(a_{i_1}x_{i_1})\ =
\ f(a\,x_{i_1})\\[.3em]
\strt\hspace{3em}\ =\ f(a)f(x_{i_1})\ =\ (f(a')+z)f(x_{i_1})
\ =\ f(a'x_{i_1})\ =\ f(a'_{i_1}x_{i_1}).$
\end{minipage}\end{equation}
Hence $a'_{i_1}\neq 0,$ i.e., $i_1\in\supp(a').$
So $I_1$ is contained in $\supp(a'),$ and so is finite.

Now let $f_1,$ respectively $f_0,$ be the maps
$A\to B$ given by projecting $A$ to its subalgebra $\prod_{I_1} A_i,$
respectively $\prod_{I-I_1} A_i,$ and then applying $f.$
Thus, these are homomorphisms satisfying $f=f_1+f_0.$
Since $\prod_{I_1} A_i$ and $\prod_{I-I_1} A_i$ annihilate
each other in either order in $A,$ the same is true of
the images of $f_1$ and $f_0$ in $B.$
Now~(\ref{d.B=}), adjusted in the light of~(\ref{d.I1}), says that
$B=f_1(A)+Z(B).$
Since $f_0(A)$ annihilates both summands in this expression,
it is contained in $Z(B),$ as claimed.
\end{proof}

We remark, in connection with the decomposition $f=f_1+f_0,$
that though the sum of two algebra homomorphisms is usually not
a homomorphism, it is if the images annihilate one another;
in particular, if one of those images is contained in the total
annihilator ideal of the codomain.
(Cf.~\cite[Lemma~4]{prod_Lie1}.)

For related results on homomorphic images of
{\em inverse limits} of nilpotent algebras $A_i,$ see \cite{pro-np},
\cite{pro-np2}.

\section{Applications to Lie algebras}\label{S.Lie}

We record in this section some consequences of Theorem~\ref{T.main}
for Lie algebras.
Note that for $B$ a Lie algebra, our definition of $Z(B)$
describes what is called the {\em center} of $B,$
and regularly denoted by that symbol.

\begin{theorem}[{cf.\ \cite[Theorems 21 and 22]{prod_Lie1}}]\label{T.svbl,np}
Let $k$ be an infinite field, let $B$ be a Lie algebra and
$(L_i)_{i\in I}$ a family of Lie algebras over $k,$ and let
$f:L=\prod_I L_i\to B$ be a surjective homomorphism of Lie algebras.
Suppose also that one of the three conditions
\textup{(i)--(iii)} of Theorem~\ref{T.main} relating
$\card(I),$ $\card(k),$ and $\dim_k(B)$ holds.

Then if all $L_i$ are {\em solvable}, respectively, {\em nilpotent},
$B$ is as well, and is in fact the sum of its center $Z(B)$ and
the \textup{(}mutually centralizing\textup{)} images under $f$ of
finitely many of the $L_i.$
\end{theorem}

\begin{proof}
The part of the conclusion after ``{\it $B$ is as well}\/'' comes
directly from Theorem~\ref{T.main}.
The preceding part follows
because a Lie algebra spanned by finitely many mutually
centralizing Lie subalgebras which are solvable or nilpotent
(as are both the $f(L_i)$ and $Z(B))$
is again solvable, respectively, nilpotent.
\end{proof}

If, instead of looking at nilpotent or solvable Lie algebras
as in Theorem~\ref{T.svbl,np},
we assume the $L_i$ simple, the situation
is not as straightforward.
Let us call a general (not necessarily Lie) algebra {\em simple}
if it has nonzero multiplication, and has no nonzero proper
homomorphic image.
A simple algebra $A$ is necessarily idempotent,
i.e., satisfies $AA = A$ (where by $AA$ we mean the span of the set
of pairwise products of elements of $A),$ and has $Z(A)=\{0\}.$
As noted in \cite[Lemma~23]{prod_Lie1}, an infinite direct
product $A$ of algebras $A_i$ which are each idempotent
is itself idempotent if and only if there is an integer $n$ such that
in all but finitely many of the $A_i,$
every element can be written as a sum of $\leq n$ products.
When no such $n$ exists, so that $AA$ is a proper ideal
of $A,$ then $A$ has the nonzero homomorphic
image $A/AA$ with zero multiplication,
and this is clearly not a direct product of simple algebras.
So for Lie algebras, in the latter situation, a description
of the general homomorphic image $B$ of $A$ under the
conditions of Theorem~\ref{T.main} must combine
a direct product of simple algebras and an abelian factor.

But do there exist {\em finite-dimensional}
simple Lie algebras that require unboundedly many
brackets to represent their general element?
Probably not.
It follows from the result of \cite{Brown}
(or \cite[Ch.VIII, \S11, Exercise~13(b)]{Bourbaki}) that in a simple
Lie algebra over $\mathbb{C}$ (or any algebraically closed field of
characteristic~$0),$ every element can be written as a single bracket.
No example seems to be
known of an element of a finite-dimensional simple
Lie algebra over any field $k$ which cannot be so represented, though
even for $k=\mathbb{R},$ the most one knows at present is that every
element is a sum of two brackets (\cite[Corollary~A3.5, p.653]{KHH+SAM},
\cite[fourth paragraph of~\S9]{prod_Lie1}).

However, using the recent
result \cite[Theorem~A]{Bois} that every finite-dimensional
simple Lie algebra over an algebraically closed field $k$
of characteristic not~$2$ or~$3$ can be {\em generated} by two elements,
we showed in \cite[Theorem~26]{prod_Lie1} that over any
infinite field $k$ (not necessarily algebraically closed)
of characteristic not~$2$ or~$3,$ every finite-dimensional
simple Lie algebra $L$ contains two elements
$x_1$ and $x_2$ such that $L=[x_1,L]+[x_2,L].$
(For related results, cf.~\cite{NN_split}.)
Combining this fact with Theorem~\ref{T.main} above, we get

\begin{theorem}[{cf.\ \cite[Theorem~27]{prod_Lie1}}]\label{T.simple}
Let $k$ be any infinite field of characteristic not~$2$ or~$3.$
Let $B$ be a Lie algebra over $k,$ $(L_i)_{i\in I}$ a family of
finite-dimensional simple Lie algebras over $k,$
and $f:L=\prod_I L_i\to B$ a surjective homomorphism of Lie algebras.
Suppose, moreover, that one of the three conditions
\textup{(i)--(iii)} of Theorem~\ref{T.main} relating
$\card(I),$ $\card(k),$ and $\dim_k(B)$ holds.

Then $f$ factors as
$\prod_I L_i\to L_{i_1}\times\dots\times L_{i_n}\cong B$
for some $i_1,\dots,i_n\in I,$
where the arrow represents the natural projection.
Thus, $B$ is finite-dimensional, and is the direct product of the
images under $f$ of finitely many of the~$L_i.$
\end{theorem}

\begin{proof}
By~\cite[Theorem~26]{prod_Lie1},
quoted above, every element of each of the $L_i$
is a sum of at most two brackets, hence the
same is true in $L,$ hence in $B.$
In particular, $B$ is idempotent (in Lie algebra
language, perfect): $B=[B,B].$
Combining this with Theorem~\ref{T.main}, we get
\begin{equation}\begin{minipage}[c]{35pc}
$B\ =\ [B,B]\ =\ [f(A_\mathrm{fin})+Z(B),\,f(A_\mathrm{fin})+Z(B)]\ =
\ [f(A_\mathrm{fin}),\,f(A_\mathrm{fin})]\ \subseteq
\ f(A_\mathrm{fin}).$
\end{minipage}\end{equation}
Hence $B=f(A_\mathrm{fin}),$ so it is a sum of the mutually
annihilating images of the simple Lie algebras $L_i;$
in particular, $Z(B)=\{0\}.$
Combining this with the final assertions of
Theorem~\ref{T.main} gives the desired conclusions.
\end{proof}

\section{Algebras over valuation rings}\label{S.KRk}

Can we extend our results to
more general commutative base rings than fields?

If $R$ is an integral domain with field of fractions $k,$ and
$f:\prod_I A_i\to B$ a homomorphism of $\!R\!$-algebras, and
we assume that the $A_i$ and $B$ are torsion-free
as $\!R\!$-modules, we might hope that
by extending scalars to the field of fractions $k$ of $R,$
and applying the preceding results to the extended map, we
could get similar conclusions about $f.$
Unfortunately, $(\prod_I A_i)\otimes_R k$ is in general much smaller
than $\prod_I (A_i\otimes_R k):$ the former can be identified
with the subalgebra of the latter consisting of those elements
whose components admit a common denominator in $R.$
So though a homomorphism $\prod_I A_i\to B$ induces a
homomorphism $(\prod_I A_i)\otimes_R k\to B\otimes_R k,$
there is no reason to expect this to extend to a homomorphism on
$\prod_I (A_i\otimes_R k),$ to which
we might apply Theorem~\ref{T.main}.

If instead we try to generalize the
results that go into the proof of Theorem~\ref{T.main}, we find
it is not hard to extend the proofs of Lemmas~\ref{L.poly},~\ref{L.rat}
and~\ref{L.JPBell} to show that the
kernel of a map from $R^I$ to a free $\!R\!$-module, or even to
a $\!k\!$-vector-space, of appropriate dimension, contains
elements with all but finitely many coordinates nonzero.
But that is not enough: the obvious analog
of Lemma~\ref{L.supp}(iii) requires $(u_i)$ to have all
but finitely many coordinates {\em invertible}.
Let us prove a slightly stronger analog of that lemma,
which uses a condition intermediate
between ``all nonzero'' and ``all invertible'', namely
``having a nonzero common multiple''.

\begin{lemma}\label{L.supp2}
Suppose $R$ is an integral domain,
$(A_i)_{i\in I}$ a family of $\!R\!$-algebras,
$B$ an $\!R\!$-algebra that is torsion-free as an $\!R\!$-module,
$f:A=\prod_{i\in I}A_i\to B$ a surjective algebra homomorphism,
and $a=(a_i)_{i\in I}$ an element of $A;$ and consider the
$\!R\!$-module homomorphism
\begin{equation}\begin{minipage}[c]{35pc}\label{d.g2}
$g_a:R^I\to B$\quad defined by\quad $g_a((u_i))=f((u_i a_i))$\quad
for all\quad $(u_i)\in R^I.$
\end{minipage}\end{equation}

Then\\[.2em]
\textup{(i)}~ If $\ker(g_a)$ contains an element
$u=(u_i)_{i\in I}$ whose entries $u_i$ have a nonzero
common multiple $r\in R,$ then $f(a)\in Z(B).$\\[.2em]
\textup{(ii)}~ More generally, given $u\in\ker(g_a)$ such that
the entries of $u$ that are nonzero admit a nonzero common multiple
$r\in R,$ if we write $a=a'+a'',$ where $\supp(a')\subseteq\supp(u)$
and $\supp(a'')\subseteq I-\supp(u),$ then $f(a')\in Z(B).$\\[.2em]
\textup{(iii)}~ Hence, if $\ker(g_a)$ contains an element
all but finitely many of whose entries are invertible,
then $a$ is the sum of an
element $a'\in f^{-1}(Z(B))$ and an element $a''\in A_\mathrm{fin}.$
\end{lemma}

\begin{proof}
(i) is proved like assertion~(i) of Lemma~\ref{L.supp}, except
that where we obtained $f(a)\,b=0$ by a computation involving
the coefficients $u_i^{-1}\in k,$ we now
prove $r\,f(a)\,b=0,$ by a computation involving
the coefficients $r\,u_i^{-1}$ (which lie in $R$ by choice of $r).$
Since $B$ is torsion-free, the relation $r\,f(a)\,b=0$
then implies $f(a)\,b=0,$ as required.
One gets $b\,f(a)=0$ in the same way, and
deduces~(ii) from~(i) as before.

To get~(iii) from~(ii) we need to see that the hypothesis of~(iii)
implies that the nonzero entries of $(u_i)$
have a nonzero common multiple.
Such a common multiple is equivalent to a nonzero common multiple
of those of the nonzero entries that are not invertible.
By assumption, there are only finitely many of these, so
they have such a common multiple (e.g., their product).
\end{proof}

We shall now show that if $R$ is a commutative valuation ring, we can,
under appropriate conditions, use some of our earlier results
on linear maps on $k^I$ to get elements $u\in R^I$ as in~(iii) above.
Because of the way we will relate
our results to those of the preceding sections, the symbol $k$
will be used below for the residue field of $R,$
rather than its field of fractions.
The following lemma will be our tool for converting to our present
goal certain of our earlier results,
namely, those that give subspaces of $k^I$ whose
nonzero members have almost all components nonzero.

\begin{lemma}\label{L.aeunit}
Let $R$ be a commutative valuation ring, $K$ its field of fractions,
$k$ its residue field, $I$ a set, and $\lambda$ a cardinal.
Suppose that $k^I$ has a subspace $W$ of dimension $\lambda,$ every
nonzero element of which has cofinite support.

Then if $g$ is an $\!R\!$-linear map from $R^I$ to a
$\!K\!$-vector-space $V$ of dimension $<\lambda,$ there
exists an element $u\in\ker(g)$ all but finitely many of
whose entries are invertible in $R.$
\end{lemma}

\begin{proof}
Let $c\mapsto\overline{c}$ be the residue map $R\to k,$ and for
$u=(u_i)\in R^I,$ let us similarly write
$\overline{u}$ for $(\,\overline{u_i}\,)\in k^I.$
If we take a basis of the
subspace $W\subseteq k^I$ indexed by $\lambda,$ and lift
its elements to $R^I,$ we get a $\!\lambda\!$-tuple of elements
$u^{(\alpha)}=(u_i^{(\alpha)})\in R^I$ $(\alpha\in\lambda)$ whose images
$\overline{u^{(\alpha)}}\in k^I$ have the property that every
nontrivial $\!k\!$-linear combination of them has all but
finitely many entries nonzero.

Since $V$ is $\!<\lambda\!$-dimensional,
some nontrivial $\!K\!$-linear relation
$\sum_{\alpha\in\lambda}c^{(\alpha)} g(u^{(\alpha)})=0$ must hold in
$V,$ where $c^{(\alpha)}\in K,$ almost all zero.
Since $R$ is a valuation ring, and all but finitely many
$c^{(\alpha)}$ are zero, we can, by multiplying by an
appropriate member of $K,$ assume that all $c^{(\alpha)}$ lie
in $R,$ but that not all lie in the maximal ideal.
Hence $\sum_{\alpha\in\lambda}
\overline{c^{(\alpha)}}\ \overline{u^{(\alpha)}}\in k^I$ is a
nontrivial linear combination of the $\overline{u^{(\alpha)}},$
so as noted, it has all but finitely many entries nonzero.
Thus $u=\sum_{\alpha\in\lambda}c^{(\alpha)} u^{(\alpha)}$ is a member of
$\ker(g)$ with all but finitely many entries invertible.
\end{proof}

We can now combine the above result with
Lemmas~\ref{L.rat} and~\ref{L.JPBell}.
(In contrast, Corollary~\ref{C.<meas} does not yield a subspace
of $k^I$ of the desired sort, and cannot be modified to do
so, for the reasons noted in the second paragraph of~\S\ref{S.I}.)

\begin{theorem}\label{T.overR}
Let $R$ be a commutative valuation ring with infinite residue
field $k,$ and $f:A=\prod_I A_i\to B$ a
homomorphism from a direct product of $\!R\!$-algebras to a
torsion-free $\!R\!$-algebra $B.$
Let us write $\mathrm{rk}_R(B)$ for the rank of $B$
as an $\!R\!$-module; i.e., the common cardinality of all
maximal $\!R\!$-linearly independent subsets of $B.$
Suppose further that either\\[.2em]
\textup{(i)}\ \ $\mathrm{rk}_R(B)<\card(k),$ and $\card(I)\leq\card(k),$
or\\[.2em]
\textup{(ii)}\,\ $\mathrm{rk}_R(B)<2^{\aleph_0},$ and
$\card(I)=\aleph_0.$\vspace{.2em}

Then, as in Theorem~\ref{T.main},
\begin{equation}\begin{minipage}[c]{35pc}\label{d.B=2}
$B\ =\ f(A_{\mathrm{fin}})+Z(B),$
\end{minipage}\end{equation}
and $f$ can be written as the sum of a homomorphism $f_1$ that
factors through the projection to a finite subproduct
of $\prod_I A_i,$ and a homomorphism $f_0$ with image in~$Z(B).$
\end{theorem}

\begin{proof}
Since $B$ is $\!R\!$-torsion free, it embeds in
the $\!K\!$-algebra $B\otimes_R K;$ so replacing $B$ with this
algebra, we may assume it a $\!K\!$-algebra.
In particular, $\mathrm{rk}_R(B)$ becomes $\dim_K(B).$

Combining Lemma~\ref{L.aeunit}
with Lemma~\ref{L.rat} in case~(i) and with
Lemma~\ref{L.JPBell} in case~(ii), we see that in
either case, every $a\in A$ satisfies the hypothesis
of Lemma~\ref{L.supp2}(iii), giving~(\ref{d.B=2}).

The final statement about the decomposition of $f$
is proved exactly as in Theorem~\ref{T.main}.
(That part of the proof, the final paragraph, does
not use the $\!k\!$-vector-space structure.)
\end{proof}

It would be interesting to know whether one can get a
version of Theorem~\ref{T.overR} under a hypothesis similar
to that of Theorem~\ref{T.main}(iii).

The following example shows that in Theorem~\ref{T.overR}, one
cannot weaken the assumption that $k$ is infinite to merely
say that $R$ is.

\begin{lemma}\label{L.DVR}
Let $R$ be a complete discrete valuation ring with finite
residue field \textup{(}e.g., the ring of $\!p\!$-adic integers
for a prime $p,$ or the formal power series ring $k[[t]]$
for $k$ a finite field\textup{)}, and let $I$ be an infinite set.

Then there exists a surjective $\!R\!$-algebra homomorphism
$R^I\to R$ whose kernel contains
$(R^I)_\mathrm{fin},$ and which therefore
does not satisfy\textup{~(\ref{d.B=2})}.
\end{lemma}

\begin{proof}[Sketch of proof]
Let $p$ be a generator of the maximal ideal of $R.$
(So in the $\!p\!$-adic example, ``$p$'' can be the
given prime $p,$ while in the formal power series example, it
can be taken to be~$t.)$

Then $R$ is the inverse limit of the system of finite rings
$R/(p^n),$ hence it admits a structure of compact topological ring.
Letting $\U$ be any nonprincipal ultrafilter on $I,$
we can take limits of $\!I\!$-tuples of elements of $R$
with respect to $\U$ under this compact topology,
and so get a ring homomorphism $f:R^I\to R$
defined by $f(a)=\lim_{\,\U}\ a_i\in R.$
(In nonstandard analysis, this is called the ``standard part map'';
cf.~\cite[p.82, Theorem~5.1]{Loeb}.)
Clearly, this homomorphism is
surjective; but for $a\in(R^I)_\mathrm{fin},$ we clearly have $f(a)=0.$
\end{proof}

Returning to the second paragraph of this section, we remark that
there actually do exist commutative rings $R$ (other than
fields) with the property that for
certain infinite cardinals $\mu,$ every $\!\mu\!$-tuple of nonzero
elements of $R$ has a nonzero common multiple; so that
if our index-set $I$ has cardinality $\leq\mu,$
we can, as proposed in that paragraph, get a result on homomorphisms
from a product of algebras $\prod_I A_i$
to a $\!R\!$-torsion-free $\!R\!$-algebra $B$
by tensoring with the field of fractions of $R$ and
applying the results of earlier sections.
For example, any nonprincipal ultraproduct of
integral domains has common multiples of countable families, and
more generally, for any cardinal $\mu,$ an ultraproduct of integral
domains with respect to a $\!\mu\!$-regular ultrafilter
\cite[4.3.2 et seq.]{Ch+Keis}
has $\!\mu\!$-fold common multiples.
A different class of examples is given by valuation rings whose
value group has cofinality $\geq\mu$ as an ordered set.

But such exotic rings $R$ are
less often used than general valuation rings.

\section{Can the cardinality assumptions of \S\S\ref{S.dimB}-\ref{S.I} be improved?}\label{S.further}

The results of \S\ref{S.dimB} and \S\ref{S.I}
give sufficient conditions on $\card(k),$
$\card(I)$ and $\dim_k(V)$ for the kernel of a map $g:k^I\to V$ to
have elements with cofinite support.
We may ask how close to optimal those results are.

For any $k$ and any nontrivial $V,$ if a statement of that sort is
to hold, it must require $\card(I)$ to be less than all
measurable cardinals $\mu>\card(k)$ (if these exist).
This is because any $I$ of cardinality greater
than or equal to such a $\mu$
admits a nonprincipal $\!\card(k)^+\!$-complete ultrafilter $\U,$
which makes $k^I/\,\U$ one-dimensional (cf.\ proof of~(\ref{d.dim1})
above, or \cite[Theorem~49]{prod_Lie1}),
and hence embeddable in $V,$
though the kernel of $k^I\to k^I/\,\U$ consists
of elements whose zero-sets lie in $\U,$ and hence are infinite.
Thus, Corollary~\ref{C.<meas}
has the weakest possible hypothesis on $I.$
However, that corollary is proved under the strong hypothesis
that $\dim_k(V)$ be finite; we don't know whether the same weak
hypothesis on $I$ can be combined with higher bounds on $\dim_k(V).$

On the other hand, fixing $k$ and an infinite set $I,$ and looking at
how large $V$ can be allowed to be, we see that for $V=k^{\aleph_0},$
projection of $k^I$ to a countable subproduct
gives a map $k^I\to V$ whose kernel has no elements
of finite support; so we cannot allow $\dim_k(V)$ to
reach $\dim_k(k^{\aleph_0}).$
By the Erd\H{o}s-Kaplansky Theorem \cite[Theorem~IX.2, p.246]{NJ},
this equals $\card(k)^{\aleph_0}.$
Now if $\card(k)$ has the form $\lambda^{\aleph_0}$
for some $\lambda,$ then $\card(k)^{\aleph_0}=\card(k);$
so in that case, Lemma~\ref{L.rat} gives the
weakest possible hypothesis on $\dim_k(V).$
Likewise, the hypothesis on $\dim_k(V)$ in
Lemma~\ref{L.JPBell} is optimal for countable~$k.$
But we don't know whether for general uncountable $k,$
we can weaken the hypothesis $\dim_k(V)<\card(k)$ of Lemma~\ref{L.rat}
all or part of the way to $\dim_k(V)<\card(k)^{\aleph_0}.$

Turning to our results on algebras over fields,
let us mention that~\cite[Theorem~19]{prod_Lie1}
combines the very weak hypothesis on $\card(I)$ in
Theorem~\ref{T.main}(iii) of this note with the
hypothesis $\dim_k(B)\leq\aleph_0,$
weaker than that of Theorem~\ref{T.main}(iii),
but imposes the additional condition
that as an algebra, $B$ satisfy ``chain
condition on almost direct factors'' (defined there).
That condition is automatic for finite-dimensional algebras,
hence that result subsumes part~(iii) of our present theorem.
We do not know whether that chain condition can be dropped from the
result of~\cite{prod_Lie1}.

Incidentally, most of the results of~\cite{prod_Lie1} do not exclude
the case where $\card(I)$ is $\geq$ a measurable cardinal
$>\card(k),$ but instead give, in that case, a conclusion in which
factorization of $f:\prod_I A_i\to B$
through finitely many of the $A_i$ is replaced by factorization
through finitely many ultraproducts of the $A_i$ with respect to
$\!\card(k)^+\!$-complete ultrafilters.
Though similar factorizations for a linear map $g:k^I\to B$
appear in Lemma~\ref{L.ultra} of this note, an apparent
obstruction to carrying these over
to results on algebra homomorphisms is that our
proof of the latter applies the results
of \S\S\ref{S.dimB}-\ref{S.I} not just to a single linear map
$g_a:k^I\to B,$ but to one such map for each $a\in A=\prod_I A_i;$
and different maps yield different families of ultrafilters.
However, one can get around this by choosing finitely many
elements $a_1,\dots,a_d\in A$ whose images under $f$ span $B,$
regarding them as together determining a map
$g_{a_1,\dots,a_d}:k^I\to B^d,$
applying Lemma~\ref{L.ultra} to that map, and then showing
that the image under $f$ of any element in the kernels of all the
resulting ultraproduct maps has zero product with the images
of $a_1,\dots,a_d\in A,$ hence lies in $Z(B).$
For the sake of brevity we have not set down formally a generalization
of Theorem~\ref{T.main}(iii) based on this argument.

For other results on cardinality and factorization of maps on
products, but of a somewhat different flavor, see~\cite{E+F+M}.


\end{document}